\documentclass[a4paper, 11pt]{article}
\usepackage{amsmath}
\usepackage{amssymb}
\usepackage{amsthm}
\usepackage{graphicx}
\usepackage{enumitem}
\usepackage[colorlinks=true, linkcolor=blue, citecolor=blue]{hyperref}
\usepackage[top=0.5in, bottom=0.5in, left=1in, right=1in]{geometry}
\usepackage{titling}
\usepackage{authblk}
\usepackage[numbers]{natbib}

\theoremstyle{plain}

\newtheorem{lemma}[]{Lemma}
\newtheorem{proposition}{Proposition}[]

\theoremstyle{definition}
\newtheorem{definition}{Definition}[]
\newtheorem{example}{Example}[]
 
\theoremstyle{remark}

\newcommand\norm[1]{\left\lVert#1\right\rVert}
\newcommand\affh{\textnormal{AffH}}
\DeclareMathOperator*{\Prob}{Prob}

\begin{document}
\title{On the computation of a non-parametric estimator by convex optimization}
\author[1]{Akshay Seshadri\thanks{Email: \texttt{akshay.seshadri@colorado.edu}}}
\affil[1]{Department of Physics, University of Colorado Boulder}
\author[2]{Stephen Becker}
\affil[2]{Department of Applied Mathematics, University of Colorado Boulder}
\date{\today}
\maketitle

\begin{abstract}
    Estimation of linear functionals from observed data is an important task in many subjects. Juditsky \& Nemirovski [The Annals of Statistics 37.5A (2009): 2278-2300] propose a framework for non-parametric estimation of linear functionals in a very general setting, with nearly minimax optimal confidence intervals. They compute this estimator and the associated confidence interval by approximating the saddle-point of a function. While this optimization problem is convex, it is rather difficult to solve using existing off-the-shelf optimization software. Furthermore, this computation can be expensive when the estimators live in a high-dimensional space. We propose a different algorithm to construct this estimator. Our algorithm can be used with existing optimization software and is much cheaper to implement even when the estimators are in a high-dimensional space, as long as the Hellinger affinity (or the Bhattacharyya coefficient) for the chosen parametric distribution can be efficiently computed given the parameters. We hope that our algorithm will foster the adoption of this estimation technique to a wider variety of problems with relative ease.
\end{abstract}

There are many situations where one wishes to estimate linear functionals of an unknown state using only observations of quantities determined by the state (i.e., indirect measurements of the state). Such a scenario is prevalent, for example, in tomography. Therefore, one would ideally like to incorporate these different measurements in such a way that not only we find a good estimate of the linear functional, but also ensure that the associated confidence interval is tight. Juditsky \& Nemirovski~\cite{juditsky2009nonparametric}, extending prior work of Donoho~\cite{DonohoOptimalRecovery}, propose an approach to this problem that constructs an estimator for a specified linear functional algorithmically, incorporating these indirect measurements. The (symmetric) confidence interval associated with their estimator is guaranteed to be nearly minimax optimal. Furthermore, once a model for the system has been specified, along with the number and type of measurements that will be recorded, their method can construct the estimator and the confidence interval even before seeing any data. This fact can be useful, for example, when one wishes to minimize the measurements that need to be performed to achieve a desired size of the confidence interval.

To construct this estimator, Juditsky \& Nemirovski~\cite{juditsky2009nonparametric} propose an algorithm that involves computing the saddle-point of a concave-convex function to a given precision. While this optimization problem is convex, solving it using standard optimization algorithms or off-the-shelf optimization software like CVX~\cite{cvx} is difficult in practice. In some cases, it is possible to extend the capabilities of these software to handle such saddle-point problems~\cite{juditsky2021well}. However, even in such situations, the fact remains that one might need to perform a high-dimensional optimization for constructing the estimator. These reasons make the estimation technique hard to use in practice.

In this study, we propose a different algorithm to compute the estimator that overcomes these difficulties. While our algorithm also approximates the saddle-point, it can be implemented using off-the-shelf optimization software currently available. Moreover, our algorithm replaces the potentially high-dimensional optimization with a much cheaper function evaluation, thus greatly reducing the computational burden. Furthermore, we write the estimator in a form that is more amenable to interpretation, which could serve as a starting point for further theoretical investigations about the estimator. This algorithm might also be useful in the extensions of Juditsky \& Nemirovski's framework~\cite{juditsky2009nonparametric} to several other problems~\cite{juditsky2020near, juditsky2020statistical}.

We begin by reviewing the estimation framework proposed by Juditsky \& Nemirovski~\cite{juditsky2009nonparametric}. We then present our approach to compute the estimator, along with the relevant proofs.

\section{Non-parametric estimation framework}
Suppose that we are given a set of ``states" $\mathcal{X} \subseteq \mathbb{R}^{d}$ that is a compact and convex subset of $\mathbb{R}^n$. We wish to estimate the linear functional $g^T x$, where $g \in \mathbb{R}^n$ is some fixed vector, while the state $x \in \mathcal{X}$ of the system is unknown to us. While we don't know $x$, we have access to a single measurement outcome determined by $x$. Measurements are modeled using random variables that assign probabilities to the possible outcomes depending on the state.

To that end, Juditsky \& Nemirovski~\cite{juditsky2009nonparametric} consider a family of random variables $\mathrm{Z}_\mu$ parameterized by $\mu \in \mathcal{M}$, where $\mathcal{M} \subseteq \mathbb{R}^m$ is some set of parameters. These random variables take values in a separable complete metric space (or Polish space) $(\Omega, \Sigma)$ equipped with a $\sigma$-finite Borel measure $\mathbb{P}$~\cite{juditsky2009nonparametric}. We assume that $\mathrm{Z}_\mu$ has a probability density $p_\mu$ with respect to this reference measure $\mathbb{P}$. The state $x \in \mathcal{X}$ determines the random variable $\mathrm{Z}_{A(x)}$ through an affine function $A\colon \mathcal{X} \to \mathcal{M}$, and we are given one outcome of this random variable for the purpose of estimation.

Our goal is to construct an estimator for $g^T x$ that uses an outcome of the random variable $\mathrm{Z}_{A(x)}$ to give an estimate. An estimator is a real-valued Borel measurable function on $\Omega$. The set of estimators $\mathcal{F}$ we are allowed to work with is any finite-dimensional vector space comprised of real-valued Borel measurable functions on $\Omega$ as long as it contains constant functions~\cite{juditsky2009nonparametric}. The mapping $\mathcal{D}(\mu) = p_\mu$ between the parameter $\mu$ and the corresponding probability density $p_\mu$ is called a parametric density family~\cite{juditsky2009nonparametric}. In order to be able to choose an appropriate estimator from $\mathcal{F}$ given that the probability density of the random variable is $p_{A(x)}$, we want the set of estimators $\mathcal{F}$ and the parametric density family $\mathcal{D}$ to interact well with each other. This gives rise to the notion of a good pair defined by Juditsky \& Nemirovski~\cite{juditsky2009nonparametric}.
\begin{definition}[Good pair]
    \label{defn:good_pair}
    We call a given pair $(\mathcal{D}, \mathcal{F})$ of parametric density family $\mathcal{D}$ and finite-dimensional space $\mathcal{F}$ of Borel functions on $\Omega$ a good pair if the following conditions hold.
    \begin{enumerate}
        \item $\mathcal{M}$ is a relatively open convex set in $\mathbb{R}^m$.
        \item Whenever $\mu \in \mathcal{M}$, we have $p_\mu(\omega) > 0$ for all $\omega \in \Omega$.
        \item Whenever $\mu, \nu \in \mathcal{M}$, $\phi(\omega) = \ln(p_\mu(\omega)/p_\nu(\omega)) \in \mathcal{F}$.
        \item Whenever $\phi \in \mathcal{F}$, the function
                \begin{equation*}
                    F_{\phi}(\mu) = \ln\left(\int_\Omega \exp\left(\phi(\omega)\right) p_\mu(\omega) \mathbb{P}(d\omega)\right)
                \end{equation*}
              is well-defined and concave in $\mu \in \mathcal{M}$.
    \end{enumerate}
\end{definition}
Any estimator $\widehat{g} \in \mathcal{F}$ is called an affine estimator (note, however, that $\widehat{g}$ need not be an affine function in general).

To judge the performance of an arbitrary estimator $\widehat{g}$, we define the $\epsilon$-risk as follows~\cite{juditsky2009nonparametric}.
\begin{definition}[$\epsilon$-risk]
    Given a confidence level $1 - \epsilon \in (0, 1)$, we define the $\epsilon$-risk associated with an estimator $\widehat{g}$ as
    \begin{equation*}
        \mathcal{R}(\widehat{g}; \epsilon) = \inf\left\{\delta: \sup_{x \in \mathcal{X}} \Prob_{\omega \sim p_{A(x)}}\left\{\omega: |\widehat{g}(\omega) - g^T x| > \delta\right\} < \epsilon\right\} ,
    \end{equation*}
    where $\omega \sim p_{A(x)}$ means that $\omega$ is sampled according to $p_{A(x)}$. The corresponding minimax optimal risk is defined as
    \begin{equation*}
        \mathcal{R}_*(\epsilon) = \inf_{\widehat{g}} \mathcal{R}(\widehat{g}; \epsilon)
    \end{equation*}
    where the infimum is taken over \textit{all} Borel functions $\widehat{g}$ on $\Omega$. Restricting to just the affine estimators, the affine risk is defined as
    \begin{equation*}
        \mathcal{R}_{\text{aff}}(\epsilon) = \inf_{\widehat{g} \in \mathcal{F}} \mathcal{R}(\widehat{g}; \epsilon) .
    \end{equation*}
\end{definition}
Intuitively, the $\epsilon$-risk of an estimator gives the smallest possible additive error of the estimator for the given confidence level, independent of the actual state of the system or the observed measurement outcome.

Optimizing over all possible estimators to find the best estimator can be computationally intractable. For this reason, Juditsky \& Nemirovski focus on constructing an affine estimator. It turns out that we don't lose much by restricting our attention to affine estimators. Indeed, Juditsky \& Nemirovski~\cite{juditsky2009nonparametric} prove that if $(\mathcal{D}, \mathcal{F})$ is a good pair, then there is an estimator $\widehat{g}_* \in \mathcal{F}$ with $\epsilon$-risk at most $\Phi_*(\ln(2/\epsilon))$, such that
\begin{align*}
    \mathcal{R}_{\text{aff}}(\epsilon) &\leq \Phi_*(\ln(2/\epsilon)) \leq \vartheta(\epsilon) \mathcal{R}_*(\epsilon) \\
    \vartheta(\epsilon) &= 2 + \frac{\ln(64)}{\ln(0.25/\epsilon)}
\end{align*}
for $\epsilon \in (0, 0.25)$.

Juditsky \& Nemirovski~\cite{juditsky2009nonparametric} propose the following method to construct the estimator $\widehat{g}_*$ and the risk $\Phi_*(\ln(2/\epsilon))$ (to a precision $2\delta > 0$).
\begin{enumerate}[leftmargin=0.2cm]
    \item For $r \geq 0$, define the function $\Phi_r\colon (\mathcal{X} \times \mathcal{X}) \times (\mathcal{F} \times \mathbb{R}_+) \to \mathbb{R}$ as
    \begin{align}
        \Phi_r(x, y;\ \phi, \alpha) = g^T x - g^T y\ +\ &\alpha \Bigg[\ln\left(\int_\Omega \exp(-\phi(\omega)/\alpha) p_{A(x)}(\omega) \mathbb{P}(d\omega)\right) \nonumber \\
                                    &\hspace{0.25cm} + \ln\left(\int_\Omega \exp(\phi(\omega)/\alpha) p_{A(y)}(\omega) \mathbb{P}(d\omega)\right)\Bigg] + 2\alpha r . \label{eqn:Phi_r}
    \end{align}
    Juditsky \& Nemirovski~\cite{juditsky2009nonparametric} show that $\Phi_r$ has the following properties. $\Phi_r$ is continuous and concave in $(x, y)$, continuous and convex in $(\phi, \alpha)$, and we have $\Phi_r \geq 0$. Furthermore, $\Phi_r$ has a well-defined saddle-point value.

    \item Denote the saddle-point value of $\Phi_r$ by $2 \Phi_*(r)$:
        \begin{equation}
            \Phi_*(r) = \frac{1}{2} \sup_{x, y \in \mathcal{X}} \inf_{\phi \in \mathcal{F}, \alpha > 0} \Phi_r(x, y; \phi, \alpha)
                      = \frac{1}{2} \inf_{\phi \in \mathcal{F}, \alpha > 0} \max_{x, y \in \mathcal{X}} \Phi_r(x, y; \phi, \alpha) \label{eqn:Phi_r_saddle_point}.
        \end{equation}

    \item Given a confidence level $1 - \epsilon \in (0.75, 1)$ and a precision $\delta > 0$, find $\phi_* \in \mathcal{F}$ and $\alpha_* > 0$ such that
        \begin{equation*}
            \max_{x, y \in \mathcal{X}} \Phi_{\ln(2/\epsilon)}(x, y; \phi_*, \alpha_*) \leq 2 \Phi_*(\ln(2/\epsilon)) + \delta.
        \end{equation*}
        This is achieved by minimizing the convex function
        \begin{equation*}
            \overline{\Phi}_{\ln(2/\epsilon)}(\phi, \alpha) = \max_{x, y \in \mathcal{X}} \Phi_{\ln(2/\epsilon)}(x, y; \phi, \alpha).
        \end{equation*}

    \item The estimator $\widehat{g}_* \in \mathcal{F}$ is then defined as
        \begin{equation}
            \widehat{g}_* = \phi_* + c \label{eqn:JN_estimator}
        \end{equation}
        where the constant $c$ is obtained by solving the optimization problem
        \begin{align}
            c &= \frac{1}{2} \max_{\chi \in X} \left[g^T x + \alpha_* \ln\left(\int_\Omega \exp\left(-\phi_* / \alpha_*\right) p_{A(x)} \mathbb{P}(d\omega)\right)\right] \nonumber \\
              &\hspace{0.3cm}- \frac{1}{2} \max_{y \in X} \left[-g^T y + \alpha_* \ln\left(\int_\Omega \exp\left(\phi_* / \alpha_*\right) p_{A(y)} \mathbb{P}(d\omega)\right)\right]. \label{eqn:JN_estimator_constant}
        \end{align}
        Then, the $\epsilon$-risk of the estimator $\widehat{g}_*$ is bounded above by $\Phi_*(\ln(2/\epsilon)) + 2\delta$.
\end{enumerate}
Given an observation $\omega$ of $\mathrm{Z}_{A(x)}$, our estimate for $g^T x$ is given by $\widehat{g}_*(\omega)$ with an additive error of $\Phi_*(\ln(2/\epsilon)) + 2\delta$ and a confidence level of $1 - \epsilon$. Note that computing the estimator $\widehat{g}_*$ requires one to perform optimization and can be time consuming, but once the estimator has been computed, the estimates can be obtained using $\widehat{g}_*$ almost instantaneously.

So far we have described how to find an estimate for $g^T x$ from one outcome of a single random variable $\mathrm{Z}_{A(x)}$. In practice, we will need to consider many outcomes of different random variables $\mathrm{Z}_{A^{(l)}(x)}$, which corresponds to $l = 1, \dotsc, L$ different types of measurement. More precisely, we are given Polish spaces $(\Omega^{(l)}, \Sigma^{(l)})$ equipped with a $\sigma$-finite Borel measure $\mathbb{P}^{(l)}$ for $l = 1, \dotsc, L$. We are also given a set of parameters $\mathcal{M}^{(l)}$ for $l = 1, \dotsc, L$. For each $l = 1, \dotsc, L$, we are given a family of random variables $\mathrm{Z}_{\mu_l}$ taking values in $\Omega^{(l)}$, where $\mu_l \in \mathcal{M}^{(l)}$. The random variable $\mathrm{Z}_{\mu_l}$ has a probability density $p^{(l)}_{\mu_l}$ with respect to the reference measure $\mathbb{P}^{(l)}$. As before, we are given affine mappings $A^{(l)}\colon \mathcal{X} \to \mathcal{M}^{(l)}$ for $l = 1, \dotsc, L$ that map the state $x \in \mathcal{X}$ of the system to a corresponding parameter. For each $l = 1, \dotsc, L$, we can choose estimators for the $l^{\text{th}}$ measurement from the set $\mathcal{F}^{(l)}$, which is a finite-dimensional vector space of real-valued Borel measurable functions on $\Omega^{(l)}$ that contains constant functions. To incorporate the outcomes of these different random variables, Juditsky \& Nemirovski~\cite{juditsky2009nonparametric} define the direct product of good pairs, which essentially constructs one large good pair from many smaller ones.
\begin{definition}[Direct product of good pairs]
    \label{defn:direct_product_good_pair}
    Considering the following quantities for $l = 1, \dotsc, L$. Let $(\Omega^{(l)}, \Sigma^{(l)})$ be a Polish space endowed with a Borel $\sigma$-finite measure $\mathbb{P}^{(l)}$. Let $\mathcal{D}^{(l)}(\mu_l) = p^{(l)}_{\mu_l}$ be the parametric density family for $\mu_l \in \mathcal{M}^{(l)}$. Let $\mathcal{F}^{(l)}$ be a finite-dimensional linear space of Borel functions on $\Omega^{(l)}$ containing constants, such that the pair $(\mathcal{D}^{(l)}, \mathcal{F}^{(l)})$ is good. Then the direct product of these good pairs $(\mathcal{D}, \mathcal{F}) = \bigotimes_{l = 1}^L (\mathcal{D}^{(l)}, \mathcal{F}^{(l)})$ is defined as follows.
    \begin{enumerate}
        \item The large space is $\Omega = \Omega^{(1)} \times \dotsb \times \Omega^{(L)}$ endowed with the product measure $\mathbb{P} = \mathbb{P}^{(1)} \times \dotsb \times \mathbb{P}^{(L)}$.
        \item The set of parameters is $\mathcal{M} = \mathcal{M}^{(1)} \times \dotsb \times \mathcal{M}^{(L)}$, and the associated parametric density family is $\mathcal{D}(\mu) = p_\mu \equiv \prod_{l = 1}^L p^{(l)}_{\mu_l}$ for $\mu =(\mu_1, \dotsc, \mu_L) \in \mathcal{M}$.
        \item The linear space $\mathcal{F}$ comprises of all functions $\phi$ defined as $\phi(\omega_1, \omega_2, \dotsc, \omega_L) = \sum_{l = 1}^L \phi^{(l)}(\omega_l)$, where $\phi^{(l)} \in \mathcal{F}^{(l)}$ and $\omega_l \in \Omega^{(l)}$ for $l = 1, \dotsc, L$.
    \end{enumerate}
\end{definition}
It can be verified that the direct product of good pairs is a good pair~\cite{juditsky2009nonparametric}.

To obtain an estimator that accounts for all the given measurement outcomes, we can apply the procedure outlined for constructing the estimator $\widehat{g}_*$ for the single outcome case to the direct product of good pairs. Specifically, if we observe $R_l$ outcomes of the random variable $\mathrm{Z}_l$ for $l = 1, \dotsc, L$, we need to minimize the function
\begin{equation}
    \overline{\Phi}_{\ln(2/\epsilon)}(\phi, \alpha) = \max_{x, y \in \mathcal{X}} \Phi_{\ln(2/\epsilon)}(x, y; \phi, \alpha) \label{eqn:JN_estimator_optimization}
\end{equation}
where
\begin{align}
    \Phi_{\ln(2/\epsilon)}(x, y; \phi, \alpha) = g^T x - g^T y\ +\ &\alpha \sum_{l = 1}^L R_l \Bigg[\ln\left(\int_{\Omega^{(l)}} \exp(-\phi^{(l)}(\omega)/\alpha) p_{A(x)}(\omega) \mathbb{P}^{(l)}(d\omega)\right) \nonumber \\
                                &\hspace{0.25cm} + \ln\left(\int_{\Omega^{(l)}} \exp(\phi^{(l)}(\omega)/\alpha) p_{A(y)}(\omega) \mathbb{P}^{(l)}(d\omega)\right)\Bigg] + 2\alpha \ln(2/\epsilon) \label{eqn:Phi_r_multiple_outcomes}
\end{align}
and $\phi = (\phi^{(1)}, \dotsc, \phi^{(L)})$. 

\subsection{Computational hurdles in finding an approximation to the saddle-point\label{secn:computation_problems_JN_saddle_point}}
We now point out several difficulties in performing the minimization of $\overline{\Phi}_{\ln(2/\epsilon)}(\phi, \alpha)$ defined in Eq.~\eqref{eqn:JN_estimator_optimization} over $\phi \in \mathcal{F}$ and $\alpha > 0$. Note that this optimization is necessary to construct an estimator using Juditsky \& Nemirovski's approach.

\begin{enumerate}
    \item When we have $L$ types of measurements $\mathrm{Z}_1, \dotsc, \mathrm{Z}_L$, $\phi$ is a vector of the form $\phi = (\phi^{(1)}, \dotsc, \phi^{(L)})$. Since each $\phi^{(l)}$ is a vector by construction, $\phi$ can be a high-dimensional vector. This can make the minimization $\inf_{\alpha > 0} \inf_{\phi \in \mathcal{F}} \overline{\Phi}_{\ln(2/\epsilon)}(\phi, \alpha)$ costly to implement.

        We illustrate through a simple example that the estimator $\phi$ can be a high-dimensional vector even with only one type of measurement, i.e., $L = 1$. The following is Example~1 from Juditsky \& Nemirovski~\cite{juditsky2009nonparametric}.
        \begin{example}[Discrete distributions]
            Suppose that the set of parameters $\mathcal{M} = \{\mu \in \mathbb{R}^n \mid \mu > 0, \sum_{i = 1}^n \mu_i = 1\}$ characterizing the observations is the relatively open standard simplex. Let $\mathcal{D}$ be the mapping $\mathcal{D}(\mu) = \mu$ for any parameter $\mu \in \mathcal{M}$. In other words, the probability distribution $p_\mu$ determined by the parameter $\mu$ is just $p_\mu = \mu$. Suppose that the set of observations is $\Omega = \{1, \dotsc, n\}$, and given a parameter $\mu \in \mathcal{M}$, the element $i \in \Omega$ is observed with a probability of $\mu_i$. We take $\Sigma$ to be the discrete $\sigma$-algebra and $\mathbb{P}$ to be the counting measure.

            Any estimator $\phi\colon \Omega \to \mathbb{R}$ is an $n$-dimensional real vector. In particular, $\mathcal{F} = \mathbb{R}^n$ is a natural choice for the set of estimators, such that $(\mathcal{D}, \mathcal{F})$ is a good pair. For any system with a large number of observations $n$, the estimator $\phi$ is a high-dimensional vector.
        \end{example}

    \item Since $\overline{\Phi}_{\ln(2/\epsilon)}$ is itself a maximum of the function $\Phi_{\ln(2/\epsilon)}$, popular gradient based methods can be difficult to use for minimizing $\overline{\Phi}_{\ln(2/\epsilon)}(\phi, \alpha)$ over $\phi \in \mathcal{F}$ and $\alpha > 0$, even when the function $\Phi_{\ln(2/\epsilon)}(x, y; \phi, \alpha)$ is smooth in $\phi$ and $\alpha$. We remark that recent progress has been made to extend the ability of CVX to handle such problems when the function $\Phi_{\ln(2/\epsilon)}$ possesses appropriate properties~\cite{juditsky2021well}. Nevertheless, solving saddle-point problems on CVX will still require more computational effort (at least internally) than our proposed algorithm, which can work with the current capabilities of CVX.

    \item In some cases, the computation of the function $\overline{\Phi}_{\ln(2/\epsilon)}(\phi, \alpha)$ can itself be costly. For example, suppose that $\mathcal{X} \subseteq \mathbb{R}^{d \times d}$ is the set of positive semidefinite matrices with a trace constraint. Even for a moderately large $d$, evaluating the function $\overline{\Phi}_{\ln(2/\epsilon)}(\phi, \alpha)$ can be costly for a given $\phi$ and $\alpha$. Therefore, we would like to reduce the number of calls to the function $\overline{\Phi}_{\ln(2/\epsilon)}(\phi, \alpha)$.
\end{enumerate}

Our approach to constructing the estimator circumvents the optimization over $\phi \in \mathcal{F}$, thereby addressing the above issues to a reasonable extent. We essentially solve the saddle point system from the other direction and then construct the estimator, but this requires new justification because the $\phi$-component of the saddle-point is not unique, and Ref.~\cite{juditsky2009nonparametric} uses a specific choice of $\phi$ that is constructed algorithmically.

\section{A different approach to constructing the estimator}
We begin by presenting our algorithm for constructing the estimator. The justification for the algorithm is given soon after.
\begin{enumerate}
    \item Given a confidence level $1 - \epsilon \in (0.75, 1)$ and a precision $\delta > 0$, find an $\alpha_* > 0$, and subsequently $x^*, y^* \in \mathcal{X}$, such that
        \begin{equation}
            2 \Phi_*(\ln(2/\epsilon)) + \delta \geq 2 \alpha_* \ln(2/\epsilon) + \left(g^T x^* - g^T y^* + 2 \alpha_* \ln(\affh(A(x), A(y)))\right)
        \end{equation}
        by approximately solving the convex optimization problem
        \begin{equation*}
            2\Phi_*(\ln(2/\epsilon)) = \inf_{\alpha > 0} \left[2 \alpha \ln(2/\epsilon) + \max_{x, y \in X} \left(g^T x - g^T y + 2 \alpha \ln(\affh(A(x), A(y)))\right)\right]
        \end{equation*}
        where
        \begin{equation}
            \affh(\mu, \nu) = \int_\Omega \sqrt{p_\mu(\omega) p_\nu(\omega)} \mathbb{P}(d\omega) \label{eqn:affH}
        \end{equation}
        is the Hellinger affinity between the distributions $p_\mu$ and $p_\nu$.
    \item Then, compute the corresponding $\phi$-component of the approximation to the saddle-point using
        \begin{equation}
            \phi_* = \frac{\alpha_*}{2} \ln\left(\frac{p_{A(x^*)}}{p_{A(y^*)}}\right) \label{eqn:canonical_phi}
        \end{equation}
    \item The estimator $\widehat{g}_*$ is then readily obtained by setting
        \begin{equation}
            \widehat{g}_* = \phi_* + \frac{1}{2} \left(g^T x^* + g^T y^*\right) \label{eqn:canonical_estimator}
        \end{equation}
        The $\epsilon$-risk of $\widehat{g}_*$ is bounded above by $\Phi_*(\ln(2/\epsilon)) + \delta$.
\end{enumerate}

Observe that our algorithm does not require one to compute the minimum over $\phi \in \mathcal{F}$, thus avoiding the complications listed in Sec.~\eqref{secn:computation_problems_JN_saddle_point}. The inner (continuous) concave maximization over $x, y \in \mathcal{X}$ can often be performed using standard optimization algorithms or off-the-shelf optimization software like CVX, assuming that $\mathcal{X}$ is computationally tractable and the Hellinger affinity given in Eq.~\eqref{eqn:affH} can be efficiently computed. The outer convex minimization is one-dimensional, and therefore, one can use standard numerical routines (like, for example, those available in \texttt{scipy.optimize.minimize\_scalar}) to solve this problem efficiently. Furthermore, the estimator given in Eq.~\eqref{eqn:canonical_estimator} is appealing from a theoretical standpoint because we have a closed-form expression in terms of the saddle-point approximation $(x^*, y^*)$ and $\alpha^*$.

A potential drawback of our algorithm is that the Hellinger affinity between $p_\mu$ and $p_\nu$ must be efficiently computable given $\mu$ and $\nu$. For commonly encountered distributions, like the discrete distribution, Poisson distribution, and the Gaussian distribution, one can find closed-form expressions for the Hellinger affinity, and therefore, $\affh(\mu, \nu)$ can be efficiently computed for appropriate choice of the parameters $\mu$ and $\nu$. Note that a similar problem exists in Juditsky \& Nemirovski's~\cite{juditsky2009nonparametric} approach, wherein
\begin{equation*}
    \int_\Omega e^{-\phi(\omega)/\alpha} p_\mu(\omega) \mathbb{P}(d\omega)
\end{equation*}
must be efficiently computable given the parameter $\mu$. In a sense, the requirement of our algorithm might be easier to satisfy because the Hellinger affinity is a well-studied quantity, and furthermore, it is independent of the set $\mathcal{F}$ of affine estimators one chooses for the problem.

We now provide justification of our algorithm. Our algorithm differs from that of Juditsky \& Nemirovski~\cite{juditsky2009nonparametric} on two fronts. One, we approximate the saddle-point by solving a different optimization problem, which is based on the expression for saddle-point value obtained in Prop.~\eqref{prop:Phi_r_saddle_point_xy_alpha}. Two, we make a choice for the $\phi$-component of the saddle-point which, as noted in the proof of Lemma~\eqref{lemma:xy_phi_saddle_point_existence}, is possible because the $\phi$-component of the saddle-point is not unique. We show in Prop.~\eqref{prop:canonical_estimator_risk_bound} that this approach still gives an estimator with the same guarantees on the $\epsilon$-risk as the estimator constructed by Juditsky \& Nemirovski's~\cite{juditsky2009nonparametric} algorithm.

We begin by proving the existence of saddle-point in $(x, y)$ and $\phi$ components for a fixed $\alpha > 0$.

\begin{lemma}
    \label{lemma:xy_phi_saddle_point_existence}
    Let $\alpha > 0$ be fixed. Then, for $r \geq 0$, the function
    \begin{align}
        \Phi^\alpha_r(x, y; \phi) = g^T x - g^T y\ +\ &\alpha \Bigg[\ln\left(\int_\Omega \exp(-\phi(\omega)/\alpha) p_{A(x)}(\omega) \mathbb{P}(d\omega)\right) \nonumber \\
                                                      &\hspace{0.25cm} + \ln\left(\int_\Omega \exp(\phi(\omega)/\alpha) p_{A(y)}(\omega) \mathbb{P}(d\omega)\right)\Bigg] + 2\alpha r \label{eqn:Phi_alpha_r}
    \end{align}
    defined on $(\mathcal{X} \times \mathcal{X}) \times \mathcal{F}$ has a saddle-point \textnormal{(}maximum in $(x, y)$ and minimum in $\phi$\textnormal{)}. Furthermore, if $(x^*, y^*)$ is the $(x, y)$ component of the saddle-point, then the $\phi$-component of the saddle-point can be expressed as
    \begin{equation*}
        \phi_* = \frac{\alpha}{2} \ln\left(\frac{p_{A(x^*)}}{p_{A(y^*)}}\right)
    \end{equation*}
\end{lemma}
\begin{proof}
    We essentially follow the proof of Thm.~2.1 in Goldenshluger \textit{et al.}~\cite{goldenshluger2015hypothesis} to show this result. We begin by noting that $\Phi^\alpha_r$ is continuous and concave in $(x, y) \in \mathcal{X} \times \mathcal{X}$ and continuous and convex in $\phi \in \mathcal{F}$. This follows from the corresponding properties of $\Phi_r$ proved in Ref.~\cite{juditsky2009nonparametric}, noting that $\Phi^\alpha_r(x,y;\phi) = \Phi_r(x, y; \phi, \alpha)$. Recall that $\mathcal{X}$ is a compact set and $\mathcal{F}$ is a finite-dimensional real vector space. Then, it follows from Sion-Kakutani theorem~\cite{sion1958general} that $\Phi^\alpha_r$ has a well-defined saddle-point value, that is
    \begin{equation*}
        \inf_{\phi \in \mathcal{F}} \max_{x, y \in \mathcal{X}} \Phi^\alpha_r(x, y; \phi) = \sup_{x, y \in \mathcal{X}} \inf_{\phi \in \mathcal{F}} \Phi^\alpha_r(x, y; \phi);
    \end{equation*}
    see Ref.~\cite{juditsky2009nonparametric} for details. Since $\Phi^\alpha_r(x, y; \phi)$ is continuous in $\phi \in \mathcal{F}$ for each $x, y \in \mathcal{X}$, we have that $\inf_{\phi \in \mathcal{F}} \Phi^\alpha(x, y; \phi)$ is upper semi-continuous in $(x, y) \in \mathcal{X} \times \mathcal{X}$. Then, since $\mathcal{X}$ is compact, the maximum of $\inf_{\phi \in \mathcal{F}} \Phi^\alpha(x, y; \phi)$ is attained in $\mathcal{X} \times \mathcal{X}$. Therefore, to show the existence of a saddle-point, it suffices to show that the minimum $\inf_{\phi \in \mathcal{F}} \Phi^\alpha(x, y; \phi)$ is attained in $\mathcal{F}$ for all $x, y \in \mathcal{X}$.

    We show this by proving that $\Phi^\alpha_r(x, y; \phi)$ is coercive in $\phi \in \mathcal{F}$ for all $x, y \in \mathcal{X}$, since $\mathcal{F}$ is a finite-dimensional vector space and $\Phi^\alpha_r(x, y; \phi)$ is continuous in $\phi \in \mathcal{X}$ for all $x, y \in \mathcal{X}$. For convenience, denote $\Phi^{x, y, \alpha}_r(\phi) = \Phi^\alpha_r(x, y; \phi)$. Recall that $\Phi^{x, y, \alpha}_r(\phi)$ is said to be coercive if $\lim_{\norm{\phi} \to \infty} \Phi^{x, y, \alpha}_r(\phi) = \infty$. Note that Goldenshluger \textit{et al.}~\cite{goldenshluger2015hypothesis} prove in Thm.~(2.1) that the function
    \begin{equation*}
        \Theta^{x, y}(\phi) = \ln\left(\int_\Omega \exp(-\phi(\omega)) p_{A(x)}(\omega) \mathbb{P}(d\omega)\right) + \ln\left(\int_\Omega \exp(\phi(\omega)) p_{A(y)}(\omega) \mathbb{P}(d\omega)\right)
    \end{equation*}
    is coercive in $\phi \in \mathcal{F}$ for all $x, y \in \mathcal{X}$. Since $\Phi^{x, y, \alpha}_r(\phi) = g^T x - g^T y + 2 \alpha r + \alpha \Theta^{x, y}(\phi/\alpha)$ for fixed $x, y \in \mathcal{X}$ and $\alpha > 0$, it follows that $\Phi^{x, y, \alpha}_r(\phi)$ is coercive in $\phi \in \mathcal{F}$. Therefore, $\Phi^\alpha_r$ has a saddle-point.

    Now, suppose that $(x^*, y^*)$ is the $(x, y)$ component of the saddle-point of $\Phi^\alpha_r$. Then, if $\phi_*$ is the $\phi$-component of the saddle-point, it minimizes the function $\Phi^\alpha_r(x^*, y^*; \phi)$, and therefore, $\phi_*/\alpha$ minimizes the function $\Theta^{x^*, y^*}(\phi/\alpha)$. In the proof of Thm~(2.1) of Goldenshluger \textit{et al.}~\cite{goldenshluger2015hypothesis} (Remark~(A.1) in particular), it is shown that every minimum of the function $\Theta^{x, y}(\phi)$ is of the form $1/2 \ln(p_{A(x)}/p_{A(y)}) + s$, where $s \in \mathbb{R}$ is a constant. Observe that $\Phi^\alpha_r$ is invariant under translations of the form $\phi \to \phi + s$ for $s \in \mathbb{R}$. Therefore, $\phi_*$ can be expressed as
    \begin{equation*}
        \frac{\phi_*}{\alpha_*} = \frac{1}{2} \ln\left(\frac{p_{A(x^*)}}{p_{A(y^*)}}\right)
    \end{equation*}
    by translating the $\phi$-component of the saddle-point with a constant, if necessary.
\end{proof}

Using this result, we express the saddle-point value of $\Phi_r$ in terms of $(x, y)$ and $\alpha$ components.
\begin{proposition}
    \label{prop:Phi_r_saddle_point_xy_alpha}
    The saddle-point value of the function $\Phi_r$ defined in Eq.~\eqref{eqn:Phi_r} can be expressed as
    \begin{equation*}
        2\Phi_*(r) = \inf_{\alpha > 0} \left[2 \alpha r + \max_{x, y \in \mathcal{X}} \left(g^T x - g^T y + 2 \alpha \ln(\affh(A(x), A(y)))\right)\right]
    \end{equation*}
    where $\affh$ is the Hellinger affinity defined in Eq.~\eqref{eqn:affH}. The quantity $\ln(\affh(\mu, \nu))$ is well-defined for every $\mu, \nu \in \mathcal{M}$.
\end{proposition}
\begin{proof}
    The function $\Phi_r$ can be written as $\Phi_r(x, y; \phi, \alpha) = \Phi^\alpha_r(x, y; \phi)$, where the function $\Phi^\alpha_r$ is defined in Eq.~\eqref{eqn:Phi_alpha_r}. Then, by Lemma~\eqref{lemma:xy_phi_saddle_point_existence}, we know that $\Phi^\alpha_r$ has a well-defined saddle-point for each $\alpha > 0$. Moreover, from the proof of Prop.~(3.1) of Ref.~\cite{juditsky2009nonparametric} (or Thm.~(2.1) of Ref.~\cite{goldenshluger2015hypothesis}), we know that
    \begin{equation*}
        \inf_{\phi \in \mathcal{F}} \Theta^{x, y}(\phi) = 2\ln(\text{AffH}(A(x), A(y)))
    \end{equation*}
    where
    \begin{equation*}
        \Theta^{x, y}(\phi) = \ln\left(\int_\Omega \exp(-\phi(\omega)) p_{A(x)}(\omega) \mathbb{P}(d\omega)\right) + \ln\left(\int_\Omega \exp(\phi(\omega)) p_{A(y)}(\omega) \mathbb{P}(d\omega)\right).
    \end{equation*}

    Noting that $\Phi_r(x, y; \phi) = 2 \alpha r + g^T x - g^T y + \alpha \Theta^{x, y}(\phi/\alpha)$, we have
    \begin{align*}
        2 \Phi_*(r) &= \inf_{\phi \in \mathcal{F}, \alpha > 0} \max_{x, y \in \mathcal{X}} \Phi_r(x, y; \phi, \alpha) \\
                    &= \inf_{\alpha > 0} \left[2 \alpha r + \inf_{\phi \in \mathcal{F}} \max_{x, y \in \mathcal{X}} \left(g^T x - g^T y + \alpha \Theta^{x, y}(\phi/\alpha)\right)\right] \\
                    &= \inf_{\alpha > 0} \left[2 \alpha r + \max_{x, y \in \mathcal{X}} \left(g^T x - g^T y + \alpha \inf_{\phi' \equiv \phi/\alpha \in \mathcal{F}} \Theta^{x, y}(\phi')\right)\right] \\
                    &= \inf_{\alpha > 0} \left[2 \alpha r + \max_{x, y \in \mathcal{X}} \left(g^T (x - y) + 2 \alpha \ln\left(\affh(A(x), A(y))\right)\right)\right]
    \end{align*}
    giving the desired result. We note that $\ln(\affh(A(x), A(y)))$ is continuous and concave on $\mathcal{X} \times \mathcal{X}$, as shown in Prop.~(3.1) of Ref.~\cite{juditsky2009nonparametric}. Observe that the maximum over $x, y \in \mathcal{X}$ is always non-negative because we obtain a value of zero when $x = y$.

    Note that $\affh(\mu, \nu) > 0$ for all $\mu, \nu \in \mathcal{M}$. This is because $\affh(\mu, \nu) = \int_\Omega \sqrt{p_\mu(\omega) p_\nu(\omega)} \mathbb{P}(d\omega)$, and since $p_\mu(\omega) > 0$ for every $\omega \in \Omega$ and $\mu \in \mathcal{M}$ by definition of a good pair, the integrand is positive. Consequently, $\ln(\affh(\mu, \nu))$, and therefore $\ln(\affh(A(x), A(y)))$, is well-defined.
\end{proof}

Therefore, we can use the above result to compute the $(x, y)$ and $\alpha$ components of the saddle-point of $\Phi_r$ to any precision $\delta > 0$. We can then construct an estimator as described in Eq.~\eqref{eqn:canonical_estimator}. In the following proposition, we show that the $\epsilon$-risk of this estimator is bounded above by $\Phi_*(\ln(2/\epsilon)) + \delta$.
\begin{proposition}
    \label{prop:canonical_estimator_risk_bound}
    Let $1 - \epsilon \in (0, 1)$ specify the confidence level, and let $\delta > 0$ be the given precision. Let $\alpha_* > 0$ be chosen such that
    \begin{equation*}
        2 \Phi_*(\ln(2/\epsilon)) + \delta \geq 2 \alpha_* \ln(2/\epsilon) + \max_{x, y \in \mathcal{X}} \left(g^T x - g^T y + 2 \alpha_* \ln(\affh(A(x), A(y)))\right)
    \end{equation*}
    Let $x^*, y^* \in \mathcal{X}$ be any points that attain the maximum in $\max_{x, y \in \mathcal{X}} \left(g^T x - g^T y + 2 \alpha_* \ln(\affh(A(x), A(y)))\right)$, and correspondingly, define
    \begin{equation*}
        \phi_* = \frac{\alpha_*}{2} \ln\left(\frac{p_{A(x^*)}}{p_{A(y^*)}}\right)
    \end{equation*}
    Then the estimator
    \begin{equation*}
        \widehat{g}_* = \phi_* + \frac{1}{2} \left(g^T x^* + g^T y^*\right)
    \end{equation*}
    has $\epsilon$-risk bounded above by $\Phi_*(\ln(2/\epsilon)) + \delta$.
\end{proposition}
\begin{proof}
    Since $x^*, y^* \in \mathcal{X}$ correspond to the maximum, we can write
    \begin{equation*}
        \max_{x, y \in \mathcal{X}} \left(g^T x - g^T y + 2 \alpha_* \ln(\affh(A(x), A(y)))\right) = g^T x^* - g^T y^* + 2 \alpha_* \ln(\affh(A(x^*), A(y^*)))
    \end{equation*}
    For the choice
    \begin{equation*}
        \phi_* = \frac{\alpha_*}{2} \ln\left(\frac{p_{A(x^*)}}{p_{A(y^*)}}\right)
    \end{equation*}
    we have
    \begin{equation*}
        2 \ln(\affh\left(A(x^*), A(y^*)\right) = \ln\left(\int_\Omega \exp(-\phi_*(\omega)/\alpha_*) p_{A(x^*)}(\omega) \mathbb{P}(d\omega)\right)
                                                        + \ln\left(\int_\Omega \exp(\phi_*(\omega)/\alpha_*) p_{A(y^*)}(\omega) \mathbb{P}(d\omega)\right)
    \end{equation*}
    Therefore, we can write
    \begin{equation*}
        2 \alpha_* \ln(2/\epsilon) + \max_{x, y \in \mathcal{X}} \left(g^T x - g^T y + 2 \alpha_* \ln(\affh(A(x), A(y)))\right) = \Phi^{\alpha_*}_{\ln(2/\epsilon)}(x^*, y^*; \phi_*)
    \end{equation*}
    where $\Phi^{\alpha_*}_{\ln(2/\epsilon)}$ is obtained using Eq.~\eqref{eqn:Phi_alpha_r}. Now, observe that we have
    \begin{equation*}
        \max_{x, y \in \mathcal{X}} \inf_{\phi \in \mathcal{F}} \Phi^{\alpha_*}_{\ln(2/\epsilon)}(x, y; \phi) = 2 \alpha_* \ln(2/\epsilon) + \max_{x, y \in \mathcal{X}} \left(g^T x - g^T y + 2 \alpha_* \ln(\affh(A(x), A(y)))\right)
    \end{equation*}
    which follows from the proof of Prop.~\eqref{prop:Phi_r_saddle_point_xy_alpha}. Therefore, $(x^*, y^*)$ and $\phi_*$ correspond to $(x, y)$ and $\phi$ components of the saddle-point of $\Phi^{\alpha_*}_{\ln(2/\epsilon)}$, respectively. Consequently, the points $x^*, y^* \in \mathcal{X}$ achieve the maximum in $\max_{x, y \in \mathcal{X}} \Phi^{\alpha_*}_{\ln(2/\epsilon)}(x, y; \phi_*)$. In particular, we have
    \begin{align}
        &\Phi^{\alpha_*}(x, y^*; \phi_*) \leq \Phi^{\alpha^*}_{\ln(2/\epsilon)}(x^*, y^*; \phi_*) \leq 2 \Phi_*(\ln(2/\epsilon)) + \delta \quad (\forall x \in \mathcal{X}) \nonumber \\
        &\Phi^{\alpha_*}(x^*, y; \phi_*) \leq \Phi^{\alpha^*}_{\ln(2/\epsilon)}(x^*, y^*; \phi_*) \leq 2 \Phi_*(\ln(2/\epsilon)) + \delta \quad (\forall y \in \mathcal{X}) \label{eqn:Phi_alpha_r_saddle_point_property}
    \end{align}
    where the last inequality follows from the definition of $\alpha_*$.

    Next, we rewrite the expression for the constant in the estimator $\widehat{g}_*$ in a convenient form. Since
    \begin{equation*}
        \int_\Omega \exp(-\phi_*(\omega)/\alpha_*) p_{A(x^*)}(\omega) \mathbb{P}(d\omega) = \int_\Omega \exp(\phi_*(\omega)/\alpha_*) p_{A(y^*)}(\omega) \mathbb{P}(d\omega)
    \end{equation*}
    holds for our choice of $\phi_*$, we can write
    \begin{align}
        c &\equiv \frac{1}{2} (g^T x^* + g^T y^*)
           = \frac{1}{2} \left[g^T x^* + \alpha_* \ln\left(\int_\Omega \exp(-\phi_*(\omega)/\alpha_*) p_{A(x^*)}(\omega) P(d\omega)\right) + \alpha_* \ln(2/\epsilon)\right] \nonumber \\
          &\hspace{3cm} - \frac{1}{2} \left[-g^T y^* + \alpha_* \ln\left(\int_\Omega \exp(\phi_*(\omega)/\alpha_*) p_{A(y^*)}(\omega) P(d\omega)\right) + \alpha_* \ln(2/\epsilon)\right] \nonumber \\
          &= \frac{1}{2} \Phi^{\alpha_*}_{\ln(2/\epsilon)}(x^*, y^*; \phi_*) - \left[-g^T y^* + \alpha_* \ln\left(\int_\Omega \exp(\phi_*(\omega)/\alpha_*) p_{A(y^*)}(\omega) P(d\omega)\right) + \alpha_* \ln(2/\epsilon)\right]
                                                                                                                                                                                    \label{eqn:canonical_constant_expression_x} \\
          &= \left[g^T x^* + \alpha_* \ln\left(\int_\Omega \exp(-\phi_*(\omega)/\alpha_*) p_{A(x^*)}(\omega) P(d\omega)\right) + \alpha_* \ln(2/\epsilon)\right] - \frac{1}{2} \Phi^{\alpha_*}_{\ln(2/\epsilon)}(x^*, y^*; \phi_*)
                                                                                                                                                                                        \label{eqn:canonical_constant_expression_y}
    \end{align}

    Note that our estimator is given as $\widehat{g}_* = \phi_* + c \in \mathcal{F}$. We define the quantity $R = \Phi_*(\ln(2/\epsilon)) + \delta$. Then, for any $x \in X$, we have
    \begin{align*}
        &g^T x + \alpha_* \ln\left(\int_\Omega \exp(-\widehat{g}_*/\alpha_*) p_{A(x)}(\omega) P(d\omega)\right) + \alpha_* \ln(2/\epsilon) \\
        &\qquad= g^T x + \alpha_* \ln\left(\int_\Omega \exp(-\phi_*(\omega)/\alpha_*) p_{A(x)}(\omega) P(d\omega)\right) + \alpha_* \ln(2/\epsilon) - c \\
        &\qquad= \Phi^{\alpha_*}_{\ln(2/\epsilon)}(x, y^*; \phi_*) - \frac{1}{2} \Phi^{\alpha_*}_{\ln(2/\epsilon)}(x^*, y^*; \phi_*) \\
        &\qquad\leq \Phi_*(\ln(2/\epsilon)) + \frac{\delta}{2} \\
        &\qquad= R - \frac{\delta}{2}
    \end{align*}
    where we used Eq.~\eqref{eqn:canonical_constant_expression_x} and Eq.~\eqref{eqn:Phi_alpha_r_saddle_point_property}. Similarly, using Eq.~\eqref{eqn:canonical_constant_expression_y} and Eq.~\eqref{eqn:Phi_alpha_r_saddle_point_property}, we find that
    \begin{equation*}
        -g^T y + \alpha_* \ln\left(\int_\Omega \exp(\widehat{g}_*/\alpha_*) p_{A(y)}(\omega) P(d\omega)\right) + \alpha_* \ln(2/\epsilon) \leq R - \frac{\delta}{2}
    \end{equation*}
    for all $y \in \mathcal{X}$. Then, dividing by $\alpha_* > 0$ and rearranging the terms, we find that
    \begin{align*}
        \ln\left(\int_\Omega \exp((g^T x - \widehat{g}_* - R)/\alpha_*) p_{A(x)}(\omega) P(d\omega)\right) &\leq \ln\left(\frac{\epsilon}{2}\right) - \frac{\delta}{2\alpha_*}
                                                                                                                                        \equiv \ln\left(\frac{\epsilon'}{2}\right) \quad (\forall x \in \mathcal{X}) \\
        \ln\left(\int_\Omega \exp((-g^T y + \widehat{g}_* - R)/\alpha_*) p_{A(y)}(\omega) P(d\omega)\right) &\leq \ln\left(\frac{\epsilon}{2}\right) - \frac{\delta}{2\alpha_*}
                                                                                                                                        \equiv \ln\left(\frac{\epsilon'}{2}\right) \quad (\forall y \in \mathcal{X})
    \end{align*}
    where $\epsilon' = \epsilon e^{-\delta/2\alpha_*} < \epsilon$. Then, from Markov's inequality, it follows that
    \begin{align*}
        \Prob_{\omega \sim p_{A(x)}}\left\{g^Tx - \widehat{g}_* - R \geq 0\right\} \leq \mathbb{E}[\exp((g^Tx - \widehat{g}_* - R)/\alpha_*)] &\leq \frac{\epsilon'}{2} \quad (\forall x \in \mathcal{X}) \\
        \Prob_{\omega \sim p_{A(y)}}\left\{-g^Ty + \widehat{g}_* - R \geq 0\right\} \leq \mathbb{E}[\exp((-g^Ty + \widehat{g}_* - R)/\alpha_*)] &\leq \frac{\epsilon'}{2} \quad (\forall y \in \mathcal{X}).
    \end{align*}
    Therefore, using the union bound, we can conclude that
    \begin{equation*}
        \Prob_{\omega \sim p_{A(x)}}\left\{|\widehat{g}_*(\omega) - g^Tx| \geq R\right\} \leq \epsilon' < \epsilon \quad (\forall x \in \mathcal{X}).
    \end{equation*}
    Therefore, the $\epsilon$-risk of the estimator $\widehat{g}_*$ is bounded above by $R = \Phi_*(\ln(2/\epsilon)) + \delta$.
\end{proof}

\end{document}